\documentclass[12pt]{amsart}
\usepackage{amsmath,amsthm,amsfonts,amssymb,verbatim,eucal}
\usepackage{url}
\usepackage[all]{xy}

\numberwithin{equation}{section}

\newtheorem{thm}[equation]{Theorem} 
\newtheorem{prop}[equation]{Proposition}
\newtheorem{lemma}[equation]{Lemma} 
\newtheorem{cor}[equation]{Corollary}
\newtheorem{example}[equation]{Example}
\newtheorem{remark}[equation]{Remark}

\DeclareMathOperator{\Alt}{Alt}
\DeclareMathOperator{\Ext}{Ext}
\DeclareMathOperator{\gr}{gr} 

\DeclareMathOperator{\im}{Im}
\DeclareMathOperator{\Ker}{Ker}
\DeclareMathOperator{\sgn}{sgn}
\DeclareMathOperator{\Sym}{Sym}

\newcommand{\NN}{\mathbb N}
\newcommand{\RR}{\mathbb R}
\newcommand{\DOT}{\setlength{\unitlength}{1pt}\begin{picture}(2.5,2)
               (1,1)\put(2,3.5){\circle*{3}}\end{picture}}

\newcommand{\Hom}{\mbox{\rm Hom\,}}

\renewcommand{\ker}{\mbox{\rm Ker\,}}

\newcommand{\ot}{\otimes}

\newcommand{\ld}{\lambda}                    
\newcommand{\cH}{\mathcal{H}}
\newcommand{\CC}{\mathbb{C}}  
\newcommand{\ZZ}{\mathbb{Z}}

\DeclareMathOperator{\codim}{codim}

\newcommand{\HH}{{\rm HH}}
\newcommand{\Wedge}{\textstyle\bigwedge}

\newcommand\bigquotient[2]{
        \mathchoice
            {
                \, \text{\raise.75ex\hbox{$#1$}\Big/\lower 2ex\hbox{$#2$}}\,%
            }
            {
                #1\,/\,#2
            }
            {
                #1\,/\,#2
            }
            {
                #1\,/\,#2
            }
    }

\begin{document}
\begin{abstract}
We investigate deformations of a skew group algebra that arise
from a finite group acting on a polynomial ring.
When the characteristic of the underlying field divides
the order of the group,
a new type of deformation emerges that does not occur 
in characteristic zero.
This analogue of Lusztig's graded
affine Hecke algebra for positive characteristic can
not be forged from the template of symplectic reflection 
and related algebras as originally crafted by Drinfeld.
By contrast, we show that in characteristic zero, for
arbitrary finite groups, a Lusztig-type deformation is always isomorphic to
a Drinfeld-type deformation. 
We fit all these deformations into a general theory, 
connecting 
Poincar\'e-Birkhoff-Witt deformations and Hochschild cohomology
when working over fields of arbitrary characteristic.
We make this connection by way of 
a double complex adapted from Guccione, Guccione, and Valqui, 
formed from 
the Koszul resolution of a polynomial ring and the bar resolution of a group algebra.
\end{abstract}
\title[PBW Deformations] 
{PBW Deformations of skew group algebras\\ in positive
characteristic}

\date{December 10, 2013.}
\author{A.\ V.\ Shepler}
\address{Department of Mathematics, University of North Texas,
Denton, Texas 76203, USA}
\email{ashepler@unt.edu}
\author{S.\  Witherspoon}
\address{Department of Mathematics\\Texas A\&M University\\
College Station, Texas 77843, USA}\email{sjw@math.tamu.edu}
\thanks{MSC-Numbers: 20C08, 20C20, 16E40}
\thanks{This material is based upon work supported by the National Science 
Foundation under Grant No.~0932078000, while the second author was in 
residence at the Mathematical Sciences Research Institute (MSRI) in 
Berkeley, California, during the Spring semester of 2013.
The first author was partially supported by NSF grant
\#DMS-1101177.
The second author was partially supported by
NSF grant 
\#DMS-1101399.}

\maketitle

\section{Introduction}

Deformations of skew group algebras in positive characteristic
provide a rich theory that differs substantially from the
characteristic zero setting, as we argue in this article.
Such deformations over fields
of characteristic zero are classical.
Lusztig used a graded version of the affine Hecke algebra
over the complex numbers
to investigate the representation theory of groups of Lie type.
His graded affine Hecke algebra is a deformation of the 
natural semi-direct product algebra, i.e., skew group algebra, arising
from a Coxeter group acting on a polynomial ring.  Lusztig \cite{Lusztig89}
described
the algebra in terms of generators and relations that
deform the action of the group but
preserve the polynomial ring.
Around the same time, Drinfeld~\cite{Drinfeld} gave a deformation
of the same skew group algebra,
for arbitrary finite groups, by instead deforming
the polynomial ring while preserving the action of the group.
In characteristic zero, Ram and the first author~\cite{RamShepler} 
showed that these two
deformations are isomorphic by producing an explicit 
conversion map.
In this paper, we generalize that result from Coxeter groups
to arbitrary finite groups
and show that any algebra modeled on Lusztig's graded affine Hecke algebra
in characteristic zero 
must arise via Drinfeld's construction.
In fact, we explain that in the nonmodular setting, when the characteristic of the 
underlying field does not divide the order of the acting group,
any deformation of the skew group algebra that changes the action
of the group must be equivalent to one that does not.
We also give explicit isomorphisms with which to convert between
algebra templates.

These results fail in positive characteristic.
Indeed, in the modular setting, the Hochschild cohomology of the 
acting group algebra is nontrivial and thus gives rise to deformations
of the skew group algebra of a completely new flavor.
Study has recently begun on the representation theory
of deformations of skew group algebras in positive characteristic
(see~\cite{BalagovicChen, BalagovicChen2, Bellamy-Martino, Brown-Changtong, Norton}
for example) and these new types of 
deformations may help in developing a unified theory.
In addition to combinatorial tools,
we rely on homological algebra to reveal and to 
understand these new types of deformations that do not appear in
characteristic zero. 

We first establish in Section~\ref{sec:quadratic} a
Poincar\'e-Birkhoff-Witt (PBW) property for nonhomogeneous 
quadratic algebras
(noncommutative) that simultaneously generalize
the construction patterns of Drinfeld and of Lusztig. 
We use the PBW property to 
describe  deformations of the skew group algebra
as nonhomogeneous quadratic algebras
in terms of generators and relations, 
defining the algebras $\cH_{\lambda,\kappa}$ that feature in the rest
of the article. Theorem~\ref{thm:PBWconditions} gives conditions on the
parameter functions $\lambda$ and $\kappa$ that are equivalent to $\cH_{\lambda,\kappa}$
being a PBW deformation. In Section~\ref{sec:LusztigyIsDrinfeldy},  we
give an explicit isomorphism between the Lusztig style and
the Drinfeld style deformations in characteristic zero.
We give examples in
Section~\ref{sec:modular} to show that such an isomorphism
can fail to exist in positive
characteristic.
In the next two sections, we recall the notion of graded deformation
and a resolution required to
make a precise connection between our deformations and Hochschild cohomology. Theorem~\ref{thm:PBWcohomologyconditions} shows that our conditions in
Theorem~\ref{thm:PBWconditions}, obtained combinatorially,  
correspond exactly to explicit homological conditions.
In Proposition~\ref{prop:deformation} 
and Theorem~\ref{thm:whichdeformations}, 
we pinpoint precisely the
class of deformations comprising
the algebras $\cH_{\lambda,\kappa}$.

Throughout this article, $k$ denotes a field (of arbitrary characteristic)
and tensor symbols without subscript denote tensor product over $k$: 
$\otimes = \otimes_k$.  (Tensor products over other rings will always
be indicated). We assume all $k$-algebras have unity and  
we use the notation $\NN=\ZZ_{\geq0}$ for the 
nonnegative integers. 

\section{Filtered quadratic algebras}\label{sec:quadratic}

Let $V$ be a finite dimensional
vector space over the field $k$.  We consider a finite group $G$ acting linearly
on $V$ and extend the action so that $G$ acts by automorphisms
on the tensor algebra (free associative $k$-algebra)
$T(V)$ and on the symmetric algebra $S(V)$ 
of $V$ over $k$.
We write $^g v$ for the action of any $g$ in $G$ on any element $v$
in $V$, $S(V)$, or $T(V)$, to distinguish from the product
of $g$ and $v$ in an algebra, and we let $V^g\subset V$ be the set
of vectors fixed pointwise by $g$.
We also identify the identity $1$ of the field $k$ with the identity group
element of $G$ in the group ring $kG$.

We consider an algebra generated by $V$ and the group
$G$ with a set of relations
that corresponds to deforming both the symmetric algebra $S(V)$ on $V$
{\em and} the action of the group $G$.
Specifically, let $\ld$ and $\kappa$ be linear parameter functions, 
$$
\kappa:V\ot V \rightarrow kG,\quad
\ld:kG\ot    V \rightarrow kG\, ,
$$
with $\kappa$ alternating.
We write $\kappa(v,w)$ for $\kappa(v\ot w)$
and $\ld(g,v)$ for $\ld(g\ot v)$ for ease
with notation throughout the article
(as $\ld$ and $\kappa$ both define bilinear functions).
Let
$\cH_{\ld,\kappa} $ 
be the associative $k$-algebra generated by a basis of $V$ together
with the group algebra $kG$, subject to the relations in $kG$ and
\begin{equation}\label{relations}
\begin{aligned}
&\text{a.} \ \ \ vw-wv=\kappa(v,w)\\
&\text{b.}\ \ \ gv- \,  ^g\! v g = \ld(g,v)
\end{aligned}
\end{equation}
for all $v,w$ in $V$ and $g$ in $G$.

Recall that the {\bf skew group algebra} of $G$ acting by automorphisms
on a $k$-algebra $R$ (for example, $R=S(V)$ or $R=T(V)$)
is the
semidirect product algebra $R\# G$: 
It is the $k$-vector space $R\otimes kG$
together with multiplication given by 
$$(r \otimes g)(s\otimes h)=r ( \,^{g}\! s)\otimes g h$$
for all $r,s$ in $R$ and $g,h$ in $G$.
To simplify notation, we write $rg$ in place of $r\ot g$.
Note that the skew group algebra $S(V)\# G$ is precisely the algebra $\cH_{0,0}$. 

We filter $\cH_{\ld,\kappa}$ by degree,  assigning
$\deg v=1$ and $\deg g =0$ for each $v$ in $V$ and $g$ in $G$.
Then the associated graded algebra 
$\gr\cH_{\ld,\kappa}$ 
is a quotient of 
its homogeneous version $S(V)\# G$
obtained by crossing out all but the
highest homogeneous part of each generating relation
defining $\cH_{\ld,\kappa}$
(see Li~\cite[Theorem~3.2]{Li2012}
or Braverman and Gaitsgory~\cite{BG}).
Recall that a ``PBW property'' on a filtered algebra
indicates that the homogeneous version (with respect
to some generating set of relations) and its 
associated graded algebra coincide.
Thus we say that $\cH_{\ld,\kappa}$ 
exhibits the {\bf Poincar\'e-Birkhoff-Witt (PBW) property}
(or is of {\bf PBW type})
when
$$
\gr(\cH_{\ld,\kappa}) \cong S(V)\# G
$$
as graded algebras.
In this case, we call $\cH_{\ld,\kappa}$ a {\bf PBW deformation}
of $S(V)\# G$.
Note that the  algebra $\cH_{\ld,\kappa}$ exhibits the PBW property
if and only if it has  basis 
$\{v_1^{i_1} v_2^{i_2} \cdots v_m^{i_m} g:
i_j \in \NN, g\in G\}$ as a $k$-vector space, where
$v_1,\ldots, v_m$ is any basis of $V$. 
In Theorem~\ref{thm:PBWconditions} below, we give necessary
and sufficient conditions on $\ld$ and $\kappa$ in order
for the PBW property to hold. 

Lusztig's graded version of the affine Hecke algebra
(see~\cite{Lusztig88, Lusztig89}) for a finite Coxeter group
is a special case of a quotient algebra $\cH_{\ld, 0}$
(with $\kappa\equiv 0$) displaying the PBW property; 
the parameter $\ld$ is defined using a simple root system
(see Section~\ref{sec:LusztigyIsDrinfeldy}).
Rational Cherednik algebras and symplectic reflection algebras
are also examples; they arise as quotient algebras $\cH_{0, \kappa}$
(with $\ld\equiv 0$) exhibiting the PBW property.

\begin{example}\label{example1:cyclic}
{\em
Let $k$ be a field of characteristic $p$
and consider a cyclic group $G$ of prime order $p$, generated by $g$.
Suppose $V=k^2$ with basis $v,w$ and let $g$ act as the matrix
$$
   \left(\begin{array}{cc} 1&1\\0&1\end{array}\right)
$$
on the ordered basis $v,w$.
Define an alternating
linear parameter function $\kappa:V\ot V\rightarrow kG$ 
by $\kappa(v,w)=g$ and a linear parameter function $\ld: kG\ot V
\rightarrow kG$ by $\ld(g^i,v)=0, \ \ld(g^i,w)=ig^{i-1}$ for $i=0,1,\ldots, p-1$.
Then the algebra $\cH_{\ld,\kappa}$ is a quotient of
a free $k$-algebra on three generators:
$$\cH_{\ld,\kappa}= k\langle v, w, g\rangle
 / (gv-vg,\ gw-vg-wg-1,\ vw-wv-g,\ g^p-1)\, .
$$
We will see that $\cH_{\lambda, \kappa}$ has the PBW property
by applying Theorem~\ref{thm:PBWconditions} below (see
Example~\ref{ex:apply-thm}). 
Alternatively, this may be checked directly by applying Bergman's
Diamond Lemma \cite{Bergman} to this small example. 
}
\end{example}

\section{Poincar\'e-Birkhoff-Witt property}
In this section,  
we establish necessary and  sufficient conditions for the PBW property
to hold for the filtered
quadratic algebra $\cH_{\ld,\kappa}$ 
defined by relations (\ref{relations}).
Note that a PBW condition on $\cH_{\ld,\kappa}$ 
implies that none of the elements in $kG$ collapse in $\cH_{\ld, \kappa}$.
This in turn gives 
conditions on $\lambda$ and $\kappa$ when
their images are expanded in terms of the elements of $G$.
For each $g$ in $G$, let 
$\lambda_g: kG\ot V\rightarrow k$ and
$\kappa_g: V\otimes V\rightarrow k$ be the $k$-linear maps for which 
$$
   \lambda(h,v)=\sum_{g\in G}\lambda_g(h,v) g,
    \ \ \ \ \ \kappa(u,v) =\sum_{g\in G} \kappa_g(u,v)g 
$$
for all $h$ in $G$ and $u,v$ in $V$. 
We now record PBW conditions on the parameters $\lambda$ and
$\kappa$.  Note that the first condition below implies that $\ld$ is determined
by its values on generators of the group $G$,
and the left hand side of the second condition measures
the failure of $\kappa$ to be $G$-invariant.
\begin{thm}\label{thm:PBWconditions}
Let $k$ be a field of arbitrary characteristic.
The algebra $\cH_{\ld,\kappa}$ exhibits the PBW property if, and only if, 
the following conditions hold
for all $g,h$ in $G$ and $u,v,w$ in $V$.
\begin{enumerate}
\item\label{cocycle21}
\, \ \rule[0ex]{0ex}{3ex}
$\ld(gh,v)=\ld(g,\, ^hv) h + g\ld(h,v)$ in $kG$.
\item\label{firstobstruction12}\ \ \rule[0ex]{0ex}{3ex}
$\kappa(\, ^gu,\, ^g v)g-g\kappa(u,v)
=
\ld\bigl(\ld(g,v),u\bigr)-\ld\bigl(\ld(g,u),v\bigr)$ in $kG$.
\item\label{cocycle12}\ \ \rule[0ex]{0ex}{3ex}
$\ld_h(g,v)(\, ^hu-\, ^gu)=\ld_h(g,u)(\, ^hv-\, ^gv)
$ in $V$.
\item\label{firstobstruction03}
\ \ \rule[0ex]{0ex}{3ex}
$\kappa_g(u,v)(^gw-w)+\kappa_g(v,w)(^gu-u)+\kappa_g(w,u)(^gv-v)=0$ in $V$.
\item\label{mixedbracket}
\ \ \rule[0ex]{0ex}{3ex}
$\ld\bigl(\kappa(u,v),w\bigr)
+\ld\bigl(\kappa(v,w),u\bigr)+\ld\bigl(\kappa(w,u),v\bigr)=0$ in $kG$.
\end{enumerate}
\end{thm}
\begin{proof}
The result follows from careful application 
of Bergman's Diamond Lemma~\cite{Bergman}.
Choose a monoid partial order on $\{v_1,\ldots, v_m, g:g\in G\}$
for some basis $v_1,\ldots, v_m$ of $V$
compatible with the reduction system (see~\cite{AlgorithmicMethods})
defined by the generating
relations~(\ref{relations}) of $\cH_{\ld, \kappa}$ which is a total order
(see~\cite[Remarks 7.1 and~7.2]{LevandovskyyShepler} for an example).
We check overlap ambiguities of the form $ghv$, $gvw, uvw$
for $g,h$ in $G$ and $u,v,w$ in $V$.
Setting ambiguities to zero (for example, see~\cite{doa})
results in the five conditions of the theorem.  We
omit details as the calculations are long but
straightforward.  \end{proof}

The conditions of Theorem~\ref{thm:PBWconditions} are quite restrictive.
The following corollaries contain some identities derived from these
conditions. These identities may be applied to narrow the possible
parameter functions $\ld, \kappa$ in examples. 
We write $\ker\kappa_g$ for the set
$\{v\in V:\kappa_g(v,w)=0\ \mbox{ for all } w \in V\}$.
The following corollary is from~\cite{GriffethShepler}.
\begin{cor}
Let $k$ be a field of arbitrary characteristic.
Suppose $\cH_{\lambda, \kappa}$ is of PBW type. Then for every $g$ in $G$,
either $\kappa_g \equiv 0$ or one of the following statements holds.
\begin{enumerate}
\item[(a)] $\codim V^g = 0$ (i.e., $g$ acts as the identity on $V$).
\item[(b)]
$\codim V^g = 1$ and $\kappa_g(v_1,v_2)=0$
for all $v_1,v_2$ in $V^g$.
\item[(c)]
$\codim V^g=2$ with
$\ker\kappa_g = V^g$\\ (and thus $\kappa_g$ is determined
by its values on a subspace complement to $V^g$).
\end{enumerate}
\end{cor}
\begin{proof}
Fix an element $g$ of $G$ that does not act as the identity on $V$. 
Suppose $\kappa_g$ is not identically zero, that is,
$\kappa_g(u,v)\neq 0$ for some
$u,v$ in $V$.  
We compare the dimensions of the kernel and image
of the transformation $T$ defined by $g-1$ acting
on $V$.  (Note that $V$ may lack a $G$-invariant
inner product, and the action of $T$ will replace classical arguments 
that involve the orthogonal complement of $V^g$.)
Condition~\ref{firstobstruction03} of Theorem~\ref{thm:PBWconditions}
implies that the image of $T$ is spanned by
$\,^gu-u$ and $\,^gv-v$ and thus
$$\codim V^g = \codim \Ker(T) = \dim \mbox{Im}(T)\leq 2. $$
Now assume $\codim V^g=1$.  
Condition~\ref{firstobstruction03} 
of Theorem~\ref{thm:PBWconditions} then implies that
$\kappa_g(v_1,v_2)=0$ whenever both $v_1,v_2$ lie in $V^g$
(just take some $w\not\in V^g$).
Now suppose $\codim V^g=2$ and thus
$^gu-u, \,^gv-v$ are linearly independent.
Then so are $u,v$, and $V=V^g\oplus ku\oplus kv$.
Then $$\kappa_g(u,w)=\kappa_g(v,w)=\kappa_g(v_1, w)=0$$ 
for any $v_1, w$ in $V^g$ 
(again by Condition~\ref{firstobstruction03})
and thus $\ker\kappa_g = V^g$ with $\kappa_g$ supported on
the $2$-dimensional space spanned by $u$ and $v$.
\end{proof}

We say $\ld(g,*)$ is {\bf supported} on a set $S\subset G$
whenever $\ld_h(g,v)=0$ for all $v$ in $V$ and $h$ not in $S$. 
\begin{cor}\label{cor:lambda}
Let $k$ be a field of arbitrary characteristic.
Suppose $\cH_{\lambda,\kappa}$ is of PBW type. Then 
for all $g$ in $G$:
\begin{enumerate}
\item
$\ld(1,*)$ is identically zero:
$\lambda(1,v)=0$ for all $v$ in $V$;
\item \label{lg-inverse} 
$\ld(g,*)$ determines $\ld(g^{-1}, *)$ by
  $$ g\lambda(g^{-1},v)= - \lambda(g, {}^{g^{-1}}v)g^{-1}\, ;$$
\item
$\ld(g,*)$ can be defined recursively:
For any $j\geq 1$,
$$
\ld(g^j,v)=\sum_{i=0}^{j-1} g^{j-1-i}\, \ld(g,\, ^{g^i}v)\, g^i\ ;
$$
\item 
$\ld(g,*)$ is supported on $h$ in $G$ with $h^{-1}g$ either
a reflection or the identity on $V$. 
If it is a reflection, then 
$\ld_h(g,v)=0$ for all $v$ on the reflecting hyperplane.
\item
If $V^g\neq V$,
$\lambda_1(g,v)=0$
unless $g$ is a reflection and $v\notin V^g$.
\end{enumerate}
\end{cor}
\begin{proof}
Setting $g=h=1$ in Theorem~\ref{thm:PBWconditions} (1), we obtain
$\lambda(1,v)=\lambda(1,v)+\lambda(1,v) ,$
and thus $\lambda(1,v)=0$.
Setting $h=g^{-1}$, we find that 
$$0=\lambda(1,v) = \lambda(g, {}^{g^{-1}}v)g^{-1} + g \lambda(g^{-1},v).$$
Recursively applying the same condition  yields the expression
for $\ld(g^j,v)$.

Now suppose
the action of group elements $g$ and $h$ on $V$ do not coincide,
so that the image of the transformation $T$ on $V$ defined by
$h-g$ has dimension at least $1$.  
Fix any nonzero $u$ in $V$ with $\, ^gu\neq\, ^hu$.
Then Condition~\ref{cocycle12} of
Theorem~\ref{thm:PBWconditions}
implies that $\ld_h(g,*)$ vanishes on $V^{h^{-1}g}=\ker(T)$.
If $\dim\im(T)\geq 2$, then there is some $w$ in $V$
with
$\, ^h u-\, ^gu$ and  $^h w-\, ^gw$ linearly independent
in $\im(T)$ and thus $\ld_h(g,u)=0$, i.e.,
$\ld_h(g,*)$ vanishes on the set complement of $V^{h^{-1}g}$ as well.
Therefore, $\ld_h(g,*)$ is identically zero unless
$\codim\ker(T)=1$ and $g^{-1}h$ is a reflection,
in which case $\ld_h(g,*)$ is supported off
the reflecting hyperplane.
The last statement of the corollary is merely the special case
when $h=1$.
\end{proof}

\begin{example}\label{ex:apply-thm}
{\em 
Applying Theorem~\ref{thm:PBWconditions}, straightforward calculations
show that the algebra $\cH_{\lambda,\kappa}$ of Example~\ref{example1:cyclic}
has the PBW property.} 
\end{example}

\section{Nonmodular setting and graded affine Hecke algebras}
\label{sec:LusztigyIsDrinfeldy}
Lusztig's graded version of the affine Hecke algebra 
is defined for a Coxeter group
$G \subseteq \text{GL}(V)$ with natural reflection representation
$\RR^m$ extended to $V=\CC^m$. This algebra deforms the
skew group relation defining $S(V)\# G$ 
and not the commutator relation defining 
$S(V)$. Indeed,
the {\bf graded affine Hecke algebra}
$\cH_{\text{gr\hphantom{.}aff}}$ 
(often just called the ``graded Hecke algebra'')
with parameter $\lambda$
is the associative $\CC$-algebra
generated by the group algebra $\CC G$ and the symmetric algebra $S(V)$
with relations
$$ gv=\,  ^gvg+\lambda(g,v) 
\quad\text{ for all }
v,w\in V\text{  and  } G $$
where
$\lambda:\CC G \otimes V \mapsto \CC$
is a $G$-invariant map (for $G$ acting on
itself by conjugation) vanishing on pairs
$(g,v)$ with $g$ a reflection
whose reflecting hyperplane contains $v$.

Graded affine Hecke algebras are often defined merely
using a class function on $G$ and a set of simple reflections
$\mathcal{S}$ generating $G$.
This allows one to record concretely 
the degrees of freedom in defining these algebras.
Indeed, since $kG$ is a subalgebra
of $\cH_{\text{gr\hphantom{.}aff}}$,
the parameter function
$\lambda$ is determined by its values on $\mathcal{S}$
(see Theorem~\ref{thm:PBWconditions}~(\ref{cocycle21}))
and since $\lambda$ is linear, it
can be expressed using
linear forms in $V^*$ defining the reflecting hyperplanes for $G$.
We fix a $G$-invariant inner product on $V$.
Then for any reflection $g$ in $\mathcal{S}$ and $v$ in $V$,
$$
\lambda(g,v)=c_g \, \frac{2 \langle v, \alpha_g\rangle}
{\hphantom{2}\langle \alpha_g, \alpha_g \rangle } 
=c_g \langle v, \alpha_g^\vee \rangle 
= c_g\, \Big(\frac{v-\, ^gv}{\alpha_g}\Big)$$
where
$c:G\rightarrow \CC$, $g\mapsto c_g$
is some conjugation invariant function 
(i.e., $c_g=c_{hgh^{-1}}$ for all $g,h$ in $G$).
Here, each $\alpha_g$ is a fixed ``root vector'' 
for the reflection $g$ (i.e., a vector in $V$ perpendicular
to the reflecting hyperplane fixed pointwise by $g$)
chosen so that conjugate reflections have root
vectors in the same orbit.
Note that the parameter $\lambda$ can be extended to
$\CC G\ot S(V)$ using Demazure-Lusztig (BGG) operators,
see~\cite{KriloffRam} for example.

Every graded affine Hecke algebra thus arises as an
algebra $\cH_{\lambda, \kappa}$ given by relations (\ref{relations}) 
with $\kappa\equiv 0$, in fact, an algebra with the PBW property
(see~\cite{Lusztig89}):
$$
\cH_{\text{gr\hphantom{,}aff}}
=\cH_{\lambda, 0}
$$
and $\cH_{\text{gr\hphantom{,}aff}}\cong S(V)\# G $
as $\CC$-vector spaces.
Ram and the first author~\cite{RamShepler} showed
that every graded affine Hecke algebra arises as a special case of
a {\bf Drinfeld Hecke algebra}, i.e.,
every $\cH_{\text{gr\hphantom{,}aff}}$
is isomorphic (as a filtered algebra) to 
$\cH_{0,\kappa}$ for some parameter $\kappa$:
$$\cH_{\lambda,0} = \cH_{\text{gr\hphantom{.}aff}}\cong \cH_{0,\kappa}\, .$$
Indeed, they constructed
an explicit isomorphism between any graded affine Hecke algebra
(for $G$ a Coxeter group) and a Drinfeld Hecke algebra $\cH_{0, \kappa}$
with parameter $\kappa$ defined in terms of the class function
$c$ and positive roots.

In this section, we prove that this result follows from a more general fact
about the algebras $\cH_{\ld,\kappa}$ for arbitrary finite groups
(not just Coxeter groups) and we give a proof valid
over arbitrary fields in the nonmodular setting:
Any $\cH_{\ld, 0}$  (deforming just the skew group relations)
with the PBW property
is isomorphic to some $\cH_{0, \kappa}$ (deforming just the relations 
for the commutative polynomial ring $S(V)$).  
In the proof, we construct an explicit isomorphism (of filtered algebras)
converting between the two types of nonhomogeneous
quadratic algebras.
In the next section, 
we give an example showing that the theorem
fails in positive characteristic.

\begin{thm}\label{LusztigyIsDrinfeldy}
Suppose $G$ acts linearly on a finite dimensional vector space $V$ over 
a field $k$
whose characteristic is coprime to $|G|$.
If the algebra
$\cH_{\lambda,0}$ exibits the PBW property
for some parameter $\lambda: kG\ot V \rightarrow kG$,
then there exists a parameter $\kappa: V\ot V\rightarrow kG$ 
with $$\cH_{0,\kappa}\cong \cH_{\ld, 0}$$ as filtered algebras.
\end{thm}
\begin{proof}
Define
$\gamma:V\rightarrow kG$ by
$$
\gamma(v)=\frac{1}{|G|}\sum_{a,b\,\in\, G} \lambda_{ab}(b,\, ^{b^{-1}}v) a\  
$$
and set 
$$\gamma_a(v) : = 
\frac{1}{|G|} \sum_{b\in G} \lambda_{ab}(b,{}^{b^{-1}}v)$$
so that $\gamma=\sum_a \gamma_a a$.
Define a parameter function $\kappa:V\ot V\rightarrow kG$ by
$$
\kappa(u,v)=\gamma(u)\gamma(v)-\gamma(v)\gamma(u)
+\lambda(\gamma(u),v)-\lambda(\gamma(v),u).
$$
Let $f:F \rightarrow \cH_{\lambda,0}$ be the algebra 
homomorphism from the associative $k$-algebra $F$ generated by
$v$ in $V$ and $kG$ to $\cH_{\lambda,0}$ defined by
$$
f(v)= v + \gamma(v) \quad\text{and}\quad f(g) = g
\quad\text{ for all } v \in V \text{ and } g \in G\, .
$$
We claim that Theorem~\ref{thm:PBWconditions} forces the
relations defining $\cH_{0,\kappa}$ to lie
in the kernel of $f$ and thus this map 
extends to a filtered algebra homomorphism
$$
f:\cH_{0,\kappa}\rightarrow \cH_{\lambda,0}\, .
$$

We first argue that Condition~\ref{cocycle12} 
of Theorem~\ref{thm:PBWconditions} 
implies that
$uv-vu-\kappa(u,v)$ is mapped to zero under $f$ for all $u,v$ in $V$.
Indeed, applying $f$ to the commutator of $u$ and $v$ gives
the commutator of $u$ and $v$ added to
the commutator of $\gamma(u)$ and $\gamma(v)$ plus cross terms.
But the commutator of $u$ and $v$ is zero in $\cH_{\lambda,0}$,
and by definition of $\kappa$, 
the commutator of $\gamma(u)$ and $\gamma(v)$
is just
$$\kappa(u,v)-\lambda(\gamma(u),v)+\lambda(\gamma(v),u).$$
We expand the cross terms:
$$
\begin{aligned}
u\gamma(v)-&\gamma(v)u-v\gamma(u)+\gamma(u)v\\
=&
\sum_{g\in G}\Big(\gamma_g(v)(u-\, ^gu)-\gamma_g(u)(v-\, ^gv)\Big) g
- \gamma_g(v)\lambda(g,u)+\gamma_g(u)\lambda(g,v)\\
=&
-\lambda(\gamma(v),u)+\lambda(\gamma(u),v)
+\sum_{g\in G}\Big(\gamma_g(v)(u-\, ^gu)-\gamma_g(u)(v-\, ^gv)\Big) g\, .
\end{aligned}
$$
We rewrite Condition~\ref{cocycle12} 
to show that the above sum over $g$ is zero:
Relabel
($h=ab, g=b,\, ^gu\mapsto u$, etc.) to obtain an
equivalent condition,
$$
0=\ld_{ab}(b,\, ^{b^{-1}} v)(u-\, ^au)
-\ld_{ab}(b,\, ^{b^{-1}} u)(v-\, ^av)\, ,
$$
and then divide by $|G|$ and sum over all $b$ in $G$ to obtain
$$
0=\gamma_a(v)(u-\, ^au)-\gamma_a(u)(v-\, ^av)
$$
for all $a$ in $G$ and $u,v$ in $V$.
Hence the commutator of $u$ and $v$ maps to $\kappa(u,v)$
in $\cH_{\lambda,0}$ under $f$, and
$uv-vu-\kappa(u,v)$ lies in the kernel of $f$.

We now argue
that Condition~\ref{cocycle21} of Theorem~\ref{thm:PBWconditions} 
implies that $cv-\, ^cvc$ lies in the kernel of $f$
for all $c$ in $G$ and $v$ in $V$, i.e.,
$$f(cv-\,^cvc)=
c\gamma(v)-\gamma(\, ^cv)c+\lambda(c,v)
$$ is zero in $\cH_{\lambda,0}$.
We expand $c\gamma(v)$ as a double sum over $a,b$ in $G$, reindex
(with $ca \mapsto a$)
and apply Corollary~\ref{cor:lambda}~(\ref{lg-inverse}).
We likewise expand $\gamma(^cv) c$ also as a double sum 
(but reindex with $ac \mapsto a$ and $b\mapsto cb$)
and obtain
$$\begin{aligned}
c\gamma(v)-\gamma(^cv)c
&=\frac{1}{|G|}\sum_{a,b}\lambda_{ab}(b,\, ^{b^{-1}}v)ca - 
\frac{1}{|G|}\sum_{a,b}\lambda_{ab}(b,\, ^{b^{-1}c}v)ac\\
&=\frac{1}{|G|}\sum_{a,b}\lambda_{c^{-1}ab}(b,\, ^{b^{-1}}v)a -
\frac{1}{|G|}\sum_{a,b}\lambda_{ab}(cb,\, ^{b^{-1}}v)a \\
&=\frac{-1}{|G|}\sum_{a,b}\lambda_{b^{-1}c^{-1}a}(b^{-1}, v)a-
\frac{1}{|G|}\sum_{a,b}\lambda_{ab}(cb,\, ^{b^{-1}}v)a\, . \\
\end{aligned}$$
But Condition~\ref{cocycle21} of Theorem~\ref{thm:PBWconditions} implies that
$$
\lambda_{b^{-1}c^{-1}a}(b^{-1},v)+\lambda_{ab}(cb,\, ^{b^{-1}}v)=\lambda_a(c,v)$$
and hence
$$
  c\gamma(v)-\gamma(^cv)c = 
- \frac{1}{|G|} \sum_{a,b\in G} \lambda_a(c,v) a = -\lambda(c,v),
$$
i.e.,
$f(cv-\, ^cv c)=0$.

Thus 
$
f: \cH_{0,\kappa}\rightarrow \cH_{\lambda,0}\, 
$
is a surjective homomorphism 
of filtered algebras.  An inverse homomorphism is easy
to define ($v\mapsto v-\gamma(v)$ and $g\mapsto g$)
and thus $f$ is a filtered algebra isomorphism
which in turn induces an isomorphism of graded algebras,
$$
\gr\cH_{0,\kappa}\cong \gr\cH_{\lambda,0}\cong S(V)\# G\, .
$$
Thus $\gr\cH_{0,\kappa}$ exhibits the PBW property as well.
\end{proof}

\section{Modular setting}\label{sec:modular} 

We now observe that in positive charactersitic,
the algebras $\cH_{\lambda,0}$
(modeled on Lusztig's graded affine Hecke algebra
but for arbitrary groups) 
do not always arise as algebras $\cH_{0,\kappa}$
(modeled on Drinfeld Hecke algebras), in
contrast to the zero characteristic setting.  
Indeed, the conclusion of Theorem~\ref{LusztigyIsDrinfeldy}
fails for Example~\ref{example1:cyclic}, as we show next.
However, allowing a more general parameter function $\kappa$, it is
possible to obtain the conclusion of Theorem~\ref{LusztigyIsDrinfeldy}
for this example (see Example~\ref{ex:generalkappa} below).

\begin{example}
\em{
Consider the cyclic $2$-group generated by a 
unipotent upper triangular matrix.  Say $G=\langle g \rangle \leq \text{GL}(V)$ 
where $g$ acts as the matrix
$$
   \left(\begin{array}{cc} 1&1\\0&1\end{array}\right)
$$
on the ordered basis $v,w$ of $V=k^2$ (see Example~\ref{example1:cyclic}) 
and $k$ is the finite field of order 
$2$. 

Set $\lambda(g^i,v)=0$ and $\lambda(g^i,w)=ig^{i-1}$ and let $\kappa$ be 
arbitrary (we are particularly interested
in the case that $\kappa$ is zero so that $\cH_{\lambda,\kappa}$ fits the pattern
of Lusztig's graded affine Hecke algebras).  
We argue that the  algebra $\cH_{\lambda,\kappa}$ is not 
isomorphic to  $\cH_{0, \kappa'}$ for {\em any} 
parameter $\kappa'$.

Indeed, suppose $f:\cH_{0,\kappa'}\rightarrow \cH_{\lambda,\kappa}$
 were an isomorphism of filtered algebras for some choice of
parameter $\kappa'$ with $\cH_{0,\kappa'}$
exhibiting the PBW property:
$$
\begin{aligned}
f:\quad
\cH_{0, \kappa'}
\hspace{23ex} 
&\longrightarrow
\ \ \ \cH_{\lambda, \kappa}\ ,\\
\bigquotient{k[v,x]\# G}{
\bigg(
{\large \substack{{gv-vg,}
\\{gw-wg-vg,}\rule{0ex}{1.5ex}
\\{vw-wv-\kappa'}\rule{0ex}{1.5ex}} }
\bigg) 
}\quad
&\longrightarrow
\quad\bigquotient{k[v,w]\# G}
{\bigg(
{\large \substack{ gv-vg,\\ gw-wg-vg-1 \rule{0ex}{1.5ex}
,\\ vw-wv-\kappa \rule{0ex}{1.5ex}}}
\bigg)},
\end{aligned}
$$
where $\kappa,\kappa'$ both lie in $kG=\{0,1,g,1+g\}$.
 Then 
$$
\begin{aligned}
f(g)&=g\, ,\\ 
f(v)&=v\gamma_1 + w\gamma_2 + \gamma_3\, , \text{and}\\
f(w)&=v\gamma_4 +w\gamma_5 + \gamma_6
\end{aligned}
$$
for some $\gamma_i$ in
$kG$.
We apply $f$ to the relation $gv-vg$ in $\cH_{0,\kappa'}$ and 
use the relations in $\cH_{\lambda,\kappa}$, the PBW property of 
$\cH_{\lambda, \kappa}$, 
and the fact that $G$ is abelian to see that $\gamma_2=0$.
Likewise, the relation $gw-wg-vg$ in $\cH_{0,\kappa'}$ implies that $\gamma_1=\gamma_5=\gamma_3 g$.
One may check that   
this gives a contradiction when we apply $f$ to the 
relation $vw-wv-\kappa'$ in $\cH_{0,\kappa'}$. 
Hence, $f$ can not be an algebra isomorphism.}
\end{example}

Next we will show that if one generalizes the possible parameter functions
$\kappa$ to include those of higher degree,
i.e., with image in $kG\oplus (V\ot kG)$, one does obtain
an isomorphism $\cH_{\lambda,\kappa}\cong \cH_{0,\kappa'}$:

\begin{example}\label{ex:generalkappa}
{\em Let $k,G,V$ be as in Example~\ref{example1:cyclic}. 
Let $\kappa'$ be the function defined by
$\kappa'(v,w) = g - vg^{-1}$. 
Let $f: \cH_{\lambda,\kappa}\rightarrow \cH_{0,\kappa'}$
be the algebra homomorphism defined by
$$
   f(g)=g, \ \ \ f(v) = v- g^{-1}, \ \ \ f(w)=w.
$$
We check that $f$ is well-defined by showing that it takes the relations of
$\cH_{\lambda,\kappa}$ to 0:
\begin{eqnarray*}
  f(gw-vg-wg-1) & = & gw-(v-g^{-1})g - wg -1 \\
     &=& gw -vg-wg \ \ \ = \ \ \ 0
\end{eqnarray*}
in $\cH_{0,\kappa'}$. Similarly we find
\begin{eqnarray*}
  f(vw-wv-g) &=& (v-g^{-1})w - w(v-g^{-1}) - g \\
  &=& vw-g^{-1}w -wv + wg^{-1} - g\\
   &=& g-vg^{-1} - g^{-1}w + wg^{-1} - g\\
  &=& g-vg^{-1} - (-v+w)g^{-1} + wg^{-1} -g  \ \ = \ \ 0
\end{eqnarray*}
in $\cH_{0,\kappa'}$. See \cite{doa} for the general theory of such algebras with
more general parameter functions $\kappa'$,
in particular, the second proof of \cite[Theorem 3.1]{doa}.
}
\end{example} 

\medskip

 One might expect there to be examples $\cH_{\lambda,\kappa}$ for
which there is no $\kappa '$ such that 
$\cH_{\lambda,\kappa}\cong \cH_{0,\kappa'}$,
even allowing more general values for $\kappa'$ as in the above example.
In the next section, we begin analysis of the relevant Hochschild cohomology,
which will clarify the distinction between the modular
and nonmodular settings, and will help to
understand better such examples. 
In the remainder of this paper, we continue to focus largely on the modular setting.

\section{Graded deformations and Hochschild cohomology}\label{sec:deformations}

We recall that for any algebra $A$ over a field $k$, 
the   Hochschild cohomology of an $A$-bimodule $M$ in degree $n$ is
$$
  \HH^n (A,M) = \Ext^n_{A^e}(A,M),
$$
where $A^e=A\ot A^{op}$, and $M$ is viewed as an $A^e$-module.
In case $M=A$, we abbreviate $\HH^n(A):= \HH^n(A,A)$. 
This cohomology can be investigated by using
the  bar resolution, that is, the following free resolution
of the $A^e$-module $A$: 
\begin{equation}\label{relative-bar}
 \cdots \stackrel{\delta_3}{\longrightarrow} A\ot A\ot A\ot A
  \stackrel{\delta_2}{\longrightarrow} A\ot A\ot A
  \stackrel{\delta_1}{\longrightarrow} A\ot A \stackrel{\delta_0}{\longrightarrow}
    A\rightarrow 0 , 
\end{equation}
where 
$$
   \delta_n (a_0\ot \cdots\ot a_{n+1}) = \sum_{i=0}^n (-1)^i a_0\ot \cdots
      \ot a_i a_{i+1}\ot \cdots\ot a_{n+1}
$$
for all $n\geq 0$ and $a_0,\ldots,a_{n+1}\in A$.
If $A$ is a graded algebra, its Hochschild cohomology inherits the grading.

A {\bf deformation of $A$ over $k[t]$} is an associative
$k[t]$-algebra on the vector space $A[t]$ whose product $*$ is determined by
\begin{equation}\label{star-formula}
    a_1 * a_2 = a_1a_2 + \mu_1(a_1\ot a_2) t + \mu_2(a_1\ot a_2) t^2 + \cdots
\end{equation}
where $a_1a_2$ is the product of $a_1$ and $a_2$ in $A$ and 
each $\mu_j: A\ot A \rightarrow A$ 
is a $k$-linear map (called the {\em $j$-th multiplication map}) 
extended to be linear over $k[t]$.
(Only finitely many terms may be nonzero 
for each pair $a_1, a_2$ in the above expansion.)
The multiplicative identity $1_A$ of $A$ is assumed to be the multiplicative
identity with respect to $*$.
(Every deformation of $A$
is equivalent to one for which $1_A$ is the multiplicative identity;
see~\cite[p.\ 43]{GerstenhaberSchack83}.)

Under the isomorphism 
$
  \Hom_{k}(A\ot  A, A)\cong \Hom_{A^e}(A\ot A \ot A\ot  A, A),
$
we may identify $\mu_1$ with a  2-cochain
on the  bar resolution (\ref{relative-bar}).
The associativity of the multiplication $*$ implies that
$\mu_1$ is a Hochschild 2-cocycle, that is, 
$$
  a_1\mu_1(a_2\ot a_3) + \mu_1(a_1\ot a_2a_3) =
     \mu_1(a_1a_2\ot a_3) + \mu_1(a_1\ot a_2)a_3
$$
for all $a_1,a_2,a_3\in A$
(equivalently $\delta_3^*(\mu_1)=0$), and that
\begin{equation}\label{obst1}
  \delta_3^*(\mu_2)(a_1\ot a_2\ot a_3) = \mu_1(\mu_1(a_1\ot a_2)\ot a_3)
                           -\mu_1(a_1\ot\mu_1(a_2\ot a_3))
\end{equation}
for all $a_1, a_2, a_3\in A$.
Note that this last equality 
can be written as $\delta_3^*(\mu_2)=\frac{1}{2}[\mu_1,\mu_1]$, where
$[  \cdot  ,  \cdot  ]$ is the Gerstenhaber bracket; 
see \cite{Gerstenhaber} for a definition of  the Gerstenhaber bracket in
arbitrary degrees.
We call a Hochschild 2-cocycle
$\mu$ for $A$ a {\bf noncommutative Poisson structure}
if the Gerstenhaber square bracket $[\mu,\mu]$ vanishes
in cohomology; the first multiplication map
$\mu_1$ of any deformation of $A$ satisfies this condition.

There are further conditions on the $\mu_j$ for higher values of $j$
implied by the associativity of $*$.
We will need just one more: 
\begin{equation}\label{obst2}
\delta_3^*(\mu_3)(a_1\ot a_2\ot a_3) = \mu_1(\mu_2(a_1\ot a_2)\ot a_3)
   -\mu_1(a_1\ot \mu_2(a_2\ot a_3)) 
\end{equation}

\vspace{-.5cm}

$$
\hspace{4.7cm} + \mu_2(\mu_1(a_1\ot a_2)\ot a_3)
  -\mu_2(a_1\ot \mu_1(a_2\ot a_3))
$$
for all $a_1,a_2,a_3\in A$. 
The right side defines the Gerstenhaber bracket $[\mu_1,\mu_2]$
as evaluated on $a_1\ot a_2\ot a_3$, and thus equation~(\ref{obst2}) may also be 
written as $\delta_3^*(\mu_3) = [\mu_1,\mu_2] $. 

Assume now that  $A$ is $\NN$-graded.
Let $t$ be an indeterminate and extend the grading
on $A$ to $A[t]$ by assigning $\deg t = 1$.
A {\bf graded deformation of $A$ over $k[t]$} 
is a deformation of $A$ over $k[t]$ that is graded, i.e., for which 
each $\mu_j:A\otimes A \rightarrow A$ in (\ref{star-formula})  
is a $k$-linear map that is homogeneous of degree $-j$.

We next show explicitly how the algebra $\cH_{\lambda,\kappa}$
arises as a graded deformation
of $S(V)\# G$ under
the assumption that it has the PBW property. 
In fact, every algebra $\cH_{\ld,\kappa}$ having the PBW property is the fiber
(at $t=1$) of some deformation $A_t$ of $S(V)\# G$
over $k[t]$, as stated in the next proposition. 
We give a converse afterwards in Theorem~\ref{thm:whichdeformations}. 
\begin{prop}\label{prop:deformation}
Let $k$ be a field of arbitrary characteristic.
If $\cH_{\ld,\kappa}$ exhibits the PBW property,
then there is a graded  deformation 
$A_t$ of $A=S(V)\# G$ 
over $k[t]$ for which 
$$
\cH_{\ld,\kappa}\cong A_t |_{t=1}
$$
and whose first and second multiplication maps $\mu_1$ and $\mu_2$ satisfy
\begin{equation}\label{k-mu}
\lambda(g,v) = \mu_1(g\ot v) -\mu_1( {}^gv\ot g), \quad\text{and}\quad 
\end{equation}
\begin{equation}\label{l-mu}
\kappa(v,w)= \mu_2(v\ot w) - \mu_2(w\ot v)
\end{equation} 
for all $v,w$ in $V$ and $g$ in $G$.
\end{prop}
\begin{proof}
Assume that $\cH_{\lambda,\kappa}$ has the PBW property. 
Let $\cH_{\lambda,\kappa,t}$ be the associative
$k[t]$-algebra generated by $kG$ and a basis of $V$, subject
to the relations
\begin{equation}\label{addt}
\begin{aligned}
   & gv - {}^gv g -\lambda(g,v) t\, , \\
   & vw - wv -\kappa(v,w) t^2 
\end{aligned}
\end{equation}
for all $v,w$ in $V$ and $g$ in $G$.

Since $\cH_{\lambda,\kappa}$ has the PBW property,
$A_t=\cH_{\lambda,\kappa,t}$ is a graded associative algebra (with $\deg t = 1$)
which is isomorphic to $S(V)\# G[t]$ as a vector space
and thus defines a graded deformation of $S(V)\# G$.
We need only verify that its first and second multiplication
maps $\mu_1$ and $\mu_2$ depend on $\kappa$ and $\lambda$ as claimed.

Let $v_1,\ldots,v_m$ be a basis of the vector space $V$, so that
the monomials $v_1^{i_1}\cdots v_m^{i_m}$ 
(with $i_1,\ldots,i_m$ ranging
over nonnegative integers)
form a basis of $S(V)$ and $S(V)\# G$ has basis
consisting of all $v_1^{i_1}\cdots v_m^{i_m} g$ 
(with $g$ ranging over elements of $G$).
We identify $\cH_{\lambda,\kappa,t}$ with $S(V)\# G[t]$ as a
$k[t]$-module by writing all elements of $\cH_{\lambda,\kappa,t}$
in terms of this basis
and then 
translate the multiplication on $\cH_{\lambda,\kappa,t}$ to a multiplication
on $S(V)\# G[t]$ (under this identification) to find $\mu_1$ and $\mu_2$
explicitly.

Denote the product on $\cH_{\lambda,\kappa,t}$ now by $*$
to distinguish from the multiplication in $S(V)\# G$.
We apply the relations in $\cH_{\lambda,\kappa,t}$ to expand the product 
$r* s$ of arbitrary
$r=v_1^{i_1}\cdots v_m^{i_m} g$ and $s=v_1^{j_1}\cdots v_m^{j_m} h$
(for $g,h$ in $G$)
uniquely as a polynomial in $t$ whose coefficients
are expressed in terms of the chosen basis:
$$
  r*s = rs + \mu_1(r\ot s)t + \mu_2(r\ot s) t^2 + \cdots + \mu_{l} (r\ot s)t^l,
$$
where $rs$ denotes the product in $S(V)\# G$. 
We set $r=v_1$ and $s=v_2$ and compare to the case
$r=v_2$ and $r=v_1$. 
Since 
$$
   v_2 * v_1 = v_1v_2 + \kappa(v_2,v_1)t^2 
   \ \ \ \mbox{ and } \ \ \ v_1*v_2 = v_1v_2,
$$
the first multiplication map $\mu_1$ vanishes
on $v_1\otimes v_2$ and on $v_2\otimes v_1$,
and so we have $v_2*v_1 - v_1*v_2 = \kappa(v_2,v_1) t^2$.  Thus
\begin{equation}
\kappa(v_2,v_1)= \mu_2(v_2\ot v_1) - \mu_2(v_1\ot v_2), 
\end{equation}
and the same holds if we exhange $v_1$ and $v_2$
as $\kappa$ is alternating.
Similarly, for any $v$ in $V$,
$   g*v =  {}^g v g + \lambda(g,v) t$
   and $ {}^gv*g = {}^gv g$
so that $g*v- {}^gv*g = \lambda(g,v)t$ and
\begin{equation}
\lambda(g,v) = \mu_1(g\ot v) -\mu_1( {}^gv\ot g) .
\end{equation} 
\end{proof}


We saw in the last proposition
that the algebras 
$\cH_{\lambda,\kappa}$ with the PBW property
correspond 
to graded deformations of
$S(V)\# G$ whose first and second multiplication maps
define the parameters $\lambda$ and $\kappa$ explicitly.
In fact, in the proof of Proposition~\ref{prop:deformation},
we see that the first multiplication maps $\mu_1$
of such deformations all vanish on $V\ot V$.
But precisely {\em which} class of  deformations of $S(V)\# G$ 
comprise the algebras $\cH_{\lambda,\kappa}$? 
We answer this question
by giving a converse to Proposition~\ref{prop:deformation}.

If we leave off the assumption, in the theorem below,  
that $V \ot V$ lies in the kernel of $\mu_1$,
then we obtain
deformations of $S(V)\# G$ generalizing the algebras
$\cH_{\lambda,\kappa}$
investigated in this article. 
For example, some of the other graded deformations of $S(V)\# G$ are
the Drinfeld orbifold algebras~\cite{doa}, 
and there may well be other possibilities 
in positive characteristic, such as analogs of  
Drinfeld orbifold algebras
which in addition deform the group action.

\begin{thm}\label{thm:whichdeformations}
Let $k$ be a field of arbitrary characteristic.
Let $A_t$ be a graded deformation of $A=S(V)\# G$
with multiplication maps $\mu_i$.  
Assume  $ V\ot V$ lies in the kernel of $\mu_1$. 
Define parameters 
$\lambda:kG\otimes V \rightarrow kG$ 
and 
$\kappa:V\otimes V\rightarrow kG$
by
$$
\lambda(g,v) = \mu_1(g\ot v) -\mu_1( {}^gv\ot g)\quad\text{and}\quad 
\kappa(v,w)= \mu_2(v\ot w) - \mu_2(w\ot v) 
$$
for all $v,w$ in $V$ and $g$ in $G$.  
Then the algebra $\cH_{\lambda,\kappa}$ defined by relations (\ref{relations}) 
exhibits the PBW property and
$$A_t|_{t=1}\cong \cH_{\lambda,\kappa}\, .$$
\end{thm}
\begin{proof}
Let $v_1,\ldots,v_m$ be a basis of $V$.
Let $F_t$ be the $k[t]$-algebra generated by $v_1,\ldots,v_m$ 
and all $g$ in $G$
subject only to the relations of $G$. 
Define a $k[t]$-linear map $f: F_t \rightarrow A_t$ by
$$
   f(x_{i_1}\cdots x_{i_n}) = x_{i_1} * \cdots * x_{i_n}
$$
for all words $x_{i_1}\cdots x_{i_n}$ in $F_t$,
where $*$ is the multiplication on $A_t$ 
and each $x_{i_j}$ is an element
of the generating set $\{v_1,\ldots,v_m\}\cup G$.
Then $f$ is an algebra homomorphism 
since the relations of $G$ hold in $A_t$,
as $A_t$ is graded.

We claim that $f$ is surjective. To see this, first note that $f(g)=g$,
$f(v)=v$, and $f(vg) = vg + \mu_1(v\ot g) t$ for all $v\in V$, $g\in G$,
so $vg = f(vg - \mu_1(v\ot g) t)$ in $A_t$. 
Now since $A_t$ is a graded deformation of $S(V)\# G$, it has $k[t]$-basis
given by monomials of the form $v_{i_1}\cdots v_{i_r} g$ 
where $g\in G$. 
We show that $v_{i_1}\cdots v_{i_r} g$ lies in the image of $f$,
by induction on $r$: 
We may assume $v_{i_2}\cdots v_{i_r}g = f(x)$ for some element $x$ of $F_t$. 
Then by writing $x$ as a linear combination of monomials, we find 
\begin{eqnarray*}
    f ( v_{i_1} x ) & = & f(v_{i_1}) * f(x) \\
   & = & v_{i_1} * v_{i_2}\cdots v_{i_r}g \\
   &=& v_{i_1}v_{i_2}\cdots v_{i_r}g + \mu_1(v_{i_1} , v_{i_2}\cdots v_{i_r} g) t +
   \mu_2(v_{i_1}, v_{i_2}\cdots v_{i_r}g) t^2 + \cdots \ .
\end{eqnarray*}
By induction, since each $\mu_j$ is of degree $-j$, each
$\mu_j(v_{i_1} , v_{i_2}\cdots v_{i_r}g)$ is in the image of $f$.
So $v_{i_1}\cdots v_{i_r}g$ is in the image of $f$, and $f$ is surjective.

Next we consider the kernel of $f$. Since each $\mu_j$ has degree $-j$, 
\begin{eqnarray*}
    f(gv) & = & g * v \ \ = \ \  g v + \mu_1(g\ot v) t,\\
   f( {}^gvg) & = & {}^gv * g \ \ = \ \ {}^gv g + \mu_1( {}^gv\ot g)t,
\end{eqnarray*}
and subtracting we obtain 
$$
   f( gv - {}^gv g) = ( \mu_1(g\ot v) - \mu_1( {}^gv \ot g)) t
$$
since in $S(V)\# G$, $gv = {}^gvg$. It follows that 
$$
    gv - {}^gv g - \lambda(g,v) t
=    gv - {}^gv g - 
\big(\mu_1(g\ot v) - \mu_1( {}^gv \ot g)\big) t
$$
lies in the kernel of $f$ for all $g\in G$, $v\in V$. 

Similarly we find that for all $v,w\in V$, since $V\ot V$ is in the kernel
of $\mu_1$ and $\mu_j$ has degree $-j$ for all $j$, 
\begin{eqnarray*}
     f(vw) &=& v * w \ \ = \ \ vw + \mu_2(v\ot w) t^2 ,\\
    f(wv) & = & w * v \ \ = \ \ wv + \mu_2(w\ot v) t^2 ,
\end{eqnarray*}
and subtracting we obtain
$$
   f(vw-wv) = (\mu_2(v\ot w) - \mu_2(w\ot v))t^2
$$
since in $S(V) \# G$, $vw = wv$.
Therefore
$$
   vw-wv - \kappa(v, w) t^2
=   vw-wv - \big(\mu_2(v\ot w) - \mu_2(w\ot v)\big) t^2
$$
lies in the kernel of $f$ for all $v,w\in V$.

Now let $I_t$ be the ideal of $F_t$ generated by all 
$$
\begin{aligned}
   & gv - {}^gvg - \lambda(g,v) t ,\\
   & vw-wv-\kappa (v,w) t^2,
\end{aligned}
$$
for $g$ in $G$ and $v,w$ in  $V$, so that $I_t\subseteq \ker f$.
We claim that $I_t = \ker f$. To see this, first note that by its
definition, $F_t/I_t$ is isomorphic to $\cH_{\lambda, \kappa, t}$,
where $\cH_{\lambda, \kappa,t}$ is the associative 
$k[t]$-algebra generated by $kG$ and a basis of $V$ subject
to the relations~(\ref{addt}).
Therefore we have a surjective homomorphism 
of graded algebras (with $\deg t=1$)
$$
  \cH_{\lambda,\kappa,t}\cong F_t/I_t \relbar\joinrel\twoheadrightarrow 
   F_t/\Ker f \cong A_t . 
$$
However, in each degree, the dimension of $\cH_{\lambda,\kappa,t}$
is at most that of $S(V)\# G[t]$, which is the dimension of $A_t$. 
This forces the dimensions to agree, so that the
surjection is also injective,  
$I_t = \ker f$, and 
$\cH_{\lambda,\kappa,t}\cong A_t $.
Thus, as filtered algebras,
$$\cH_{\lambda,\kappa}=\cH_{\lambda,\kappa,t}\,|_{t=1}
\cong
A_t\, | _{t=1}\, . $$ 
In addition, $\cH_{\lambda,\kappa}$
has the PBW property,
since the associated graded algebra of the fiber of
$A_t$ at $t=1$ is just $A=S(V)\# G$ (see, for example,
\cite[Section~1.4]{BG}).
\end{proof}

\begin{remark}{\em 
Note that Theorem~\ref{thm:whichdeformations}
generalizes a result for Drinfeld Hecke algebras
(also called graded Hecke algebras) in characteristic
zero to a larger class of algebras existing
over fields of arbitrary charcteristic.
In characteristic zero, Drinfeld Hecke algebras correspond
exactly to those deformations $A_t$ of $S(V)\# G$ whose
$i$-th multiplication map has graded degree $-2i$
(up to isomorphism) by~\cite[Proposition~8.14]{SheplerWitherspoon1}
(see also~\cite[Theorem~3.2]{Witherspoon} for a more general statement).
If we replace the deformation parameter $t$ by $t^2$,
we can convert each deformation $A_t$
to a graded deformation (whose odd index multiplication
maps all vanish), obtaining a special
case of Theorem~\ref{thm:whichdeformations} with
$\ld\equiv 0$.  See Corollary~\ref{gradedHeckealgebras}.}
\end{remark}

In the next section, we develop the homological algebra needed to 
make more precise the connections among the maps 
$\lambda, \kappa, \mu_1$, and $\mu_2$
and to express the conditions of Theorem~\ref{thm:PBWconditions} homologically.
We will identify $\lambda$ and $\kappa$  with cochains on a
particular resolution $X_{\DOT}$ of
$S(V)\# G$, and show that $\lambda = \phi^*_2(\mu_1)$ and 
$\kappa=\phi^*_2(\mu_2)$ for a choice of chain 
map $\phi_{\DOT}$ from $X_{\DOT}$ to 
the bar resolution of $S(V)\# G$.

\section{A resolution and chain maps}\label{sec:resolution}
We investigate deformations of $S(V)\# G$ using homological
algebra by implementing a practical resolution  
for handling Hochschild
cohomology over fields of arbitrary characteristic, 
as well as well-behaved chain maps between this resolution
and the bar resolution. 
The resolution $X_{\DOT}$ of 
$A=S(V)\# G$ we describe next is from~\cite{PBW}, 
a generalization of  a resolution  given 
by Guccione, Guccione, and Valqui \cite[\S4.1]{GGV}.
It arises as a tensor product of the Koszul resolution of $S(V)$ with the
bar resolution of $kG$, and we recall only the outcome
of this construction here. 
For details, see \cite{PBW}. 
In the case that the characteristic of $k$ does not divide the order
of $G$, one may simply use the Koszul resolution of $S(V)$ to obtain
homological information about deformations, as in that case the 
Hochschild cohomology of $S(V)\# G$ is isomorphic to 
the $G$-invariants
in the Hochschild cohomology of $S(V)$ with coefficients in $S(V)\# G$.
See, e.g., \cite{SheplerWitherspoon1}. 

The complex $X_{\DOT}$ is  the total complex 
of the double complex $X_{\DOT,\DOT}$,
where
\begin{equation}\label{xij2}
X_{i,j} := A\ot (kG)^{\ot i} \ot \Wedge^j(V)\ot A ,
\end{equation}
and $A^e$ acts by left and right multiplication on the outermost
tensor factors $A$. 
The horizontal differentials are defined by
$$
\begin{aligned}
 d_{i,j}^h(1 \ot  g_1 \ot &\cdots\ot g_i\ot x \ot 1)\\
   =\ & g_1\ot\cdots\ot g_i\ot x\ot 1 
  + \sum_{l=1}^{i-1} (-1)^l \ot g_1\ot \cdots\ot g_l g_{l+1}\ot \cdots\ot g_i\ot x\ot 1\\
   &\hspace{0cm} + (-1)^i \ot g_1\ot \cdots\ot g_{i-1}\ot  {}^{g_i}x\ot g_i
\end{aligned}
$$
for all $g_1,\ldots, g_i$ in $G$ and $x$ in $\Wedge^j(V)$, 
and the vertical differentials are defined by
$$ 
\begin{aligned}
d^v_{i,j}(1\ot g_1\ot &\cdots\ot g_i\ot v_1\wedge\cdots\wedge v_j\ot 1) \\
=\ & (-1)^i
  \sum_{l=1}^j (-1)^l ( {}^{g_1\cdots g_i}v_l\ot g_1\ot \cdots\ot g_i\ot v_1\wedge\cdots\wedge
  \hat{v}_l\wedge \cdots\wedge v_j\ot 1 \\
   & \hspace{0cm} - 1\ot g_1\ot \cdots\ot g_i\ot v_1\wedge\cdots\wedge \hat{v}_l\wedge
   \cdots\wedge v_j\ot v_l ),
\end{aligned}
$$
for all $g_1,\ldots,g_i$ in $G$ and $v_1,\ldots, v_j$ in $V$.
Setting $X_n = \oplus_{i+j=n} X_{i,j}$ for each $n\geq 0$
yields the total complex $X_{\DOT}$: 
\begin{equation}\label{resolution-X}
  \cdots\rightarrow X_2\rightarrow X_1\rightarrow X_0\rightarrow A\rightarrow 0,
\end{equation}
 with differential in degree $n$ given by 
$d_n= \sum_{i+j=n} (d_{i,j}^h + d_{i,j}^v)$,
and in degree 0 by the multiplication map. 
It is proven in \cite{PBW}, in a more general context, that $X_{\DOT}$ is a free resolution
of the $A^e$-module $A$.

We will need to extend some results on chain maps found in 
\cite{PBW}. In the following lemma, we let $A^{\ot (\DOT +2)}$
denote the bar resolution of $A$. We consider elements of $\Wedge^j(V)$
to have degree $j$ and elements of $(kG)^{\ot i}$ to have degree 0. 

\begin{lemma}\label{lem:maps}
Let $k$ be a field of arbitrary characteristic.
There exist chain maps $\phi_{\DOT}: X_{\DOT}\rightarrow A^{\ot (\DOT +2)}$
and $\psi_{\DOT}: A^{\ot (\DOT + 2)}\rightarrow X_{\DOT}$ that are of degree 0
as maps and for which $\psi_n \phi_n$ is the identity map on 
$X_n$ for $n=0,1,2$. 
\end{lemma}

As a start, 
recall that there is an embedding of the Koszul resolution into the bar resolution
of $S(V)$. 
The chain maps $\phi_{\DOT}$, $\psi_{\DOT}$ in the lemma
may be taken to be extensions of such maps on $X_{0, \DOT}$ and $S(V)^{\ot (\DOT + 2)}$.
Indeed, 
it is proven in~\cite[Lemma~4.7]{PBW},
for more general Koszul algebras, that there are chain maps 
$\phi_{\DOT}: X_{\DOT}\rightarrow A^{\ot (\DOT + 2)}$
and $\psi_{\DOT} : A^{\ot (\DOT + 2)}\rightarrow X_{\DOT}$ 
of degree 0 (as maps) for which $\psi_n\phi_n$ is the identity
on the subspace $X_{0,n}$ of $X_n$ for each $n\geq 0$.
See~\cite{SW-ANT}, for example, for explicit choices of chain maps
involving the Koszul and bar resolutions of $S(V)$.

\begin{proof}
We  show that  the maps may be chosen so that
$\psi_0\phi_0$, $\psi_1\phi_1$, and $\psi_2\phi_2$ are identity maps
on all of $X_0$, $X_1$, and $X_2$, respectively.
In degree 0, $\psi_0$ and $\phi_0$ are defined to be identity
maps from $A\ot A$ to itself.
It is noted in the proof of \cite[Lemma 4.7]{PBW} that one may choose
$\phi_1$ and $\psi_1$
to be defined by 
$$
\begin{aligned}
\phi_1(1\ot g\ot 1) &= 1\ot g\ot 1,
&\quad
\phi_1(1\ot v\ot 1) &= 1\ot v\ot 1,
\\
\psi_1(1\ot g\ot 1) &= 1\ot g\ot 1,
&\quad
\psi_1(1\ot v\ot 1) &= 1\ot v\ot 1, 
\\
\hphantom{xxxxx}
\text{and}\quad\quad\quad\, \ \ \
 & \psi_1 (1 \ot vg\ot 1)\ =&  1 \ot v\ot g + v \ot g \ot  1  \, , &
\end{aligned}
$$ 
for all $g$ in $G$ and $v$ in $V$.
Extend $\psi_1$ to $A^{\ot 3}$ by extending these elements to a free 
$A^e$-basis and defining $\psi_1$ on the remaining basis elements
arbitrarily subject to the condition $d_1\psi_1 = \psi_0 \delta_1$. 
It follows that $\psi_1\phi_1$ is the identity map on $X_1$. 

We now proceed to degree 2.
We define $\phi_2:X_2\rightarrow A^{\ot 4}$
by  setting 
\begin{equation}\label{phi2} 
\begin{aligned}
  \phi_2(1\ot g\ot h\ot 1)&=1\ot g\ot h\ot 1\, ,\\
\phi_2(1\ot g\ot v\ot 1)&= 1\ot g\ot v\ot 1 - 1\ot {}^gv\ot g\ot 1 \, , \quad\text{and}\\
  \phi_2(1\ot v\wedge w \ot 1) &= 1\ot v\ot w\ot 1 - 1\ot w\ot v\ot 1
\end{aligned}
\end{equation}
for all $g,h$ in $G$ and $v,w$ in $V$.
One may check directly that $\phi_2$ continues the chain map $\phi_{\DOT}$ to degree 2.
We construct a map $\psi_2:A^{\ot 4}\rightarrow X_2$
by first setting
\begin{equation}\label{psi2gv}
\begin{aligned}
  \psi_2(1\ot g\ot h\ot 1)&=1\ot g\ot h\ot 1  \text{ in } X_{2,0}\, ,\\
  \psi_2(1\ot g\ot v\ot 1 - 1\ot {}^gv\ot g\ot 1)&=
1\ot g\ot v\ot 1  \text{ in } X_{1,1}     \, , \\
  \psi_2(1\ot v\ot g\ot 1)& = 0 \, ,
\quad\text{and} \\
  \psi_2(1\ot v\ot w\ot 1-1\ot w\ot v\ot 1)&
= 1\ot v\wedge w\ot 1 \text{ in } X_{0,2} 
\end{aligned}
\end{equation}
for all $v,w$ in $V$ and $g$ in $G$.
(In fact, we set $\psi_2(1\ot v_i\ot v_j\ot 1)$ equal to
$1\ot v_i\wedge v_j\ot 1$ if $i<j$ and zero if $j<i$
given a fixed basis $v_1, \ldots, v_m$ of $V$.)
The elements of the form $1\ot g\ot v\ot 1 - 1\ot {}^gv\ot g\ot 1$
and $1\ot v\ot g\ot 1$, where
$v$ runs over a basis of $V$ and $g$ runs over all elements of $G$,
are linearly independent, and the same is true if we take their union with
all $1\ot v\ot w\ot 1$ where $v,w$ run over a basis of $V$ and all
$1\ot g\ot h\ot 1$ where $g,h$ run over the elements of $G$.  
We extend to a free $A^e$-basis of $A^{\ot 4}$ and define
$\psi_2$ arbitrarily on the remaining basis elements
subject to the  condition $d_2\psi_2=\psi_1\delta_2$. 
By construction, $\psi_2\phi_2$ is the identity
map on  $X_2$.
\end{proof}

\begin{remark}{\em 
An explicit choice of chain map $\phi_3$
is often helpful in checking PBW conditions homologically.
We may define $\phi_3$ on a free $A^e$-basis
of $X_{3}$: Let $g,h,\ell\in G$, $v,w, v_1,v_2,v_3\in V$, and set
\begin{equation}
\label{phi3gvw}
\begin{aligned}
\phi_3(1\ot g\ot h\ot \ell\ot 1) = \ & \phantom{ , , , }1\ot g\ot h\ot\ell\ot 1\, ,\\
\phi_3(1\ot g\ot h\ot v\ot 1) = \ & \phantom{ , , , }1\ot g\ot h\ot v\ot 1 -1\ot g\ot {}^hv\ot h\ot 1\\ 
      &  + 1\ot {}^{gh}v\ot g\ot h\ot 1\, ,\\
  \phi_3(1\ot g\ot v\wedge w\ot 1)  = \ 
     & \phantom{ , , ,  }1\ot g\ot v\ot w\ot 1\ \, - \ 1\ot g\ot w\ot v\ot 1 \\  
 & + 1\ot  {}^gv\ot  {}^gw\ot g\ot 1 -   1\ot  {}^gw\ot  {}^gv \ot g\ot 1 \\
   & -1\ot  {}^gv \ot g \ot w \ot 1\ +\,  1\ot  {}^gw\ot g\ot v \ot 1\, , \\
\phi_3(1\ot v_1\wedge v_2\wedge v_3 \ot 1) =  & \phantom{ , , , }\sum_{\sigma\in \Sym_3}
     \sgn \sigma \ot v_{\sigma(1)}\ot v_{\sigma(2)}\ot v_{\sigma(3)}\ot 1\, .
\end{aligned}
\end{equation}
It can be checked that $\delta_3\phi_3 = \phi_2  d_3$ on $X_3$. 
}
\end{remark}

In the remainder of this article, we will always choose to work with chain maps
$\phi_{\DOT}$, $\psi_{\DOT}$ satisfying the above conditions and formulas.

\section{Homological conditions arising from the PBW property}
\label{sec:homologicalconditions}
We now reverse the flow of information
and use homological theory to explain which quotient algebras
$\cH_{\ld,\kappa}$, defined by relations (\ref{relations}), 
exhibit the PBW property.
We show that the PBW conditions of Theorem~\ref{thm:PBWconditions}  
have a natural and elegant
homological interpretation.

We saw in Section~\ref{sec:deformations} that
if the algebra $\cH_{\ld,\kappa}$ has the PBW property, 
then it corresponds to a graded deformation of $S(V)\# G$
with first multiplication maps $\mu_1$ and $\mu_2$
related to the parameters $\ld$ and $\kappa$
(see Proposition~\ref{prop:deformation}).
Since the multiplication maps $\mu_1$ and $\mu_2$ 
lift to a deformation, say with third multiplication map $\mu_3$,
they satisfy
the following homological conditions:
\begin{equation}\label{barhomologicalconditions}
\begin{aligned}
1)\ \ \
\delta^*(\mu_1)&=0,
\hphantom{xxxxxxxxxxxxxxxxxxxxxxxxxxxxxxxxx}\\
2)\ \
[\mu_1, \mu_1]&=2\delta^*(\mu_2),\ \ \text{and}
\hphantom{xxxxxxxxxxxxxxxxxxxxxxxxxxxxxxx}\\
3)\ \ 
[\mu_1,\mu_2]&= \delta^* \mu_3
\hphantom{xxxxxxxxxxxxxxxxxxxxxxxxxxxxxx}
\end{aligned}
\end{equation}
as cochains (on the bar complex) 
for the Hochschild cohomology of $S(V)\# G$, where $[ \, , \, ]$
denotes the Gerstenhaber bracket 
(see Section~\ref{sec:deformations}).
These general homological conditions 
are necessary, but not sufficient,
for some random maps to arise as the multiplication maps 
of a deformation of an algebra;
the conditions merely
guarantee that the initial obstructions to
constructing a deformation from some candidate multiplication maps
vanish.

We seek homological conditions which completely characterize
the algebras $\cH_{\lambda,\kappa}$  with the PBW property.
If we begin with {\em arbitrary parameters} $\kappa$
and $\lambda$, can we find homological conditions analgous
to those above which
are both necessary {\em and sufficient} 
to imply the parameters define a deformation of $S(V)\# G$?
We begin with a lemma that shows how to extend parameters
$\ld$ and $\kappa$ to the $A^e$-resolution $X_{\DOT}$ of $A=S(V)\# G$
defined in Section~\ref{sec:resolution}.
\begin{lemma}\label{extendparameters}
${}_{}$
\begin{itemize}
\item[(1)] Any $k$-linear parameter function
$\ld: kG \ot V\rightarrow kG$
extends uniquely to a graded 2-cochain on $X_{\DOT}$
of degree $-1$ that vanishes on $X_{0,2}$.\\
\item[(2)] Any alternating $k$-linear parameter function
$\kappa: V\otimes V \rightarrow kG$
extends uniquely to graded 2-cochain on $X_{\DOT}$
of degree $-2$.
\end{itemize}
\end{lemma}
\begin{proof}
The second degree term of the complex $X_{\DOT}$
is $X_2=X_{2,0}\oplus X_{1,1}\oplus X_{0,2}$.
The $k$-linear map $\ld: kG\ot V \rightarrow kG$
extends to an $A$-bimodule map $\ld'$ on the summand $X_{1,1}=A\ot kG\ot V\ot A$ of $X_2$,
$$\ld': A\ot kG\ot V \ot A \rightarrow kG\, ,$$
which may be extended
to a map on the rest of $X_2$ by setting it to zero on $X_{2,0}$ and $X_{0,2}$.
Thus $\ld'$ is a cochain on $X_{\DOT}$ of 
degree $-1$ that extends $\ld$:
$$
\ld(g,v) = \ld'(1\ot g \ot v \ot 1)$$
for all $g$ in $G$ and $v$ in $V$.
Any other degree $-1$ cochain on $X_2$ that extends $\ld$ in this
way and also vanishes on $X_{0,2}$
would define the same cochain
(since it would automatically vanish on $X_{2,0}$ given its degree).

Likewise, the alternating $k$-linear map $\kappa: V\ot V \rightarrow kG$
naturally defines an $A$-bimodule map $\kappa'$ on the summand 
$\displaystyle{X_{0,2}=A\ot V\wedge V\ot A}$ of $X_2$,
$$\kappa': A\ot V\wedge V \ot A \rightarrow kG\, , $$
which may be extended
to a map on the rest of $X_2$ by setting it to zero on $X_{2,0}$ and $X_{1,1}$.
Thus $\kappa'$ is a cochain on $X_{\DOT}$ of 
degree $-2$ that extends $\kappa$:
$$
\kappa(v,w) = \kappa'(1\ot v \ot w \ot 1)$$
for all $v,w$ in $V$.
Any other degree $-2$ cochain on $X_2$ that extends $\kappa$ in this
way would define the same cochain
since it must vanish on $X_{2,0}$ and $X_{1,1}$ given its degree.
\end{proof}

We now express the PBW conditions of Theorem~\ref{thm:PBWconditions} as 
homological conditions on the parameters $\lambda,\kappa$. In fact,
the conditions of the theorem are concisely articulated 
using the graded Lie bracket on Hochschild cohomology.
We show that if $\cH_{\lambda,\kappa}$ has the PBW property, then
$\lambda$ defines a noncommutative Poisson structure
on $S(V)\# G$ and the Gerstenhaber bracket of 
$\lambda$ and $\kappa$ is zero in cohomology.  The converse also  
holds, and we will  make this statement  precise at the cochain level.
The bracket in the following theorem is the Gerstenhaber bracket 
(graded Lie bracket) on
Hochschild cohomology induced on the complex $X_{\DOT}$ from the bar complex
where it is defined (via the chain maps $\phi_{\DOT}$, $\psi_{\DOT}$):

\begin{thm}\label{thm:PBWcohomologyconditions}
Let $k$ be a field of arbitrary characteristic.
A quotient algebra $\cH_{\ld,\kappa}$ defined by relations (\ref{relations}) 
exhibits the PBW property if and only if
\begin{itemize}
\item
$d^*(\lambda)=0$,
\item
$[\lambda,\lambda]=2d^*(\kappa)$, and
\item
$[\lambda,\kappa]= 0 $
\end{itemize}
as Hochschild cochains, 
where $\lambda$ and $\kappa$ are viewed as 
cochains on the resolution $X_{\DOT}$ via Lemma~\ref{extendparameters}.
\end{thm}
\begin{proof}
We choose chain maps $\phi_{\DOT}$, $\psi_{\DOT}$ 
satisfying the conditions
of Lemma~\ref{lem:maps} and equations (\ref{phi2}) and (\ref{psi2gv}). 
We convert $\lambda$ and $\kappa$ to cochains on the bar
resolution of $A=S(V)\# G$
by simply applying these  chain maps.
Explicitly, we set 
$$\mu_1=\psi_2^*(\lambda)\quad\text{and}\quad\mu_2=\psi_2^*(\kappa). $$
Since $\psi_2\phi_2$ is the identity map by Lemma~\ref{lem:maps}, 
we find that $\lambda=\phi_2^*(\mu_1)$
and $\kappa=\phi_2^*(\mu_2)$, so that by formula (\ref{phi2}),
$$
\begin{aligned}
\lambda(g,v) &= \mu_1(g\ot v) -\mu_1( {}^gv\ot g)  \, ,\\ 
\kappa(v,w)&= \mu_2(v\ot w) - \mu_2(w\ot v)\\
\end{aligned} 
$$
for all $v,w$ in $V$ and $g$ in $G$.

The Gerstenhaber bracket on $X_{\DOT}$ is 
lifted from the bar complex using the chain maps $\psi$ and $\phi$
as well:
$$
[\lambda,\lambda]=\phi^*[\psi^*(\lambda),\psi^*(\lambda)]
\text{ and }
[\lambda,\kappa]=\phi^*[\psi^*(\lambda),\psi^*(\kappa)]\ .
$$

We show next that the conditions in the theorem
are equivalent to the PBW conditions of Theorem~\ref{thm:PBWconditions}.

\subsection{The Hochschild cocycle
condition on $\lambda$}
We first show that $\lambda$ is a cocycle
if and only if Conditions~\ref{cocycle21} and \ref{cocycle12} 
of Theorem~\ref{thm:PBWconditions} hold. 
The cochain $\lambda$ on the complex $X_{\DOT}$ 
is only nonzero on the summand $X_{1,1}$ of $X_2$.
We therefore
check the values of $\lambda$ on the image of the differential on 
$X_{2,1}\oplus X_{1,2}$. 

Note that $X_{2,1}$ is spanned by all
$1\ot g\ot h\ot v\ot 1$ where $g,h$ range over all pairs of
group elements in $G$ and $v$ ranges over vectors in $V$.
Since $\lambda$ is 0 on $X_{2,0}$,
$$
\begin{aligned}
   d^*(\lambda)&(1\ot g\ot h\ot v\ot 1) \\
  &= \lambda(g\ot h\ot v\ot 1 - 1\ot gh\ot v\ot 1 + 1\ot g\ot {}^hv\ot h)\\
   &= g\lambda(h,v) - \lambda(gh,v) + \lambda(g, {}^hv) h .
\end{aligned}
$$
Thus $d^*(\lambda)$ vanishes on $X_{2,1}$ exactly when
Condition~\ref{cocycle21}  of Theorem~\ref{thm:PBWconditions}
holds. 

Next we note that $X_{1,2}$ is spanned by all $1\ot g\ot v\wedge w\ot 1$,
where $g$ ranges over all group elements and $v,w$ range over vectors 
in $V$.
Then, since $\lambda$ is 0 on $X_{0,2}$, 
$$
\begin{aligned}
d^*(\lambda)&(1\ot  g\ot v\wedge w\ot 1)  \\
  &=- \lambda( {}^gv\ot g\ot w\ot 1 - 1\ot g\ot w\ot v 
     - {}^g w \ot g\ot v \ot 1 + 1\ot g\ot v \ot w)\\
   &=  -\, {}^gv \lambda(g,w) + \lambda(g,w)v + {}^gw\lambda(g,v)
   -\lambda(g,v) w .
\end{aligned}
$$
Expand $\lambda(g,v)$ as $\sum_{h\in G} \lambda_h(g,v)h$ and rearrange
terms to see that $d^*(\lambda)$ vanishes on $X_{1,2}$
exactly when Condition~\ref{cocycle12} of
Theorem~\ref{thm:PBWconditions} holds.
Thus $\lambda$ is a cocycle exactly when both
Conditions~\ref{cocycle21} and \ref{cocycle12} hold.

\subsection{The first obstruction
to integrating}
We next show that the condition
$$[\lambda,\lambda]=2 d^* \kappa$$ 
corresponds
to (\ref{obst1}),
the condition that the first obstruction to lifting 
(or integrating) an infinitesimal deformation
must vanish.
Note that this condition
gives an equality of maps on $X_3$.
But both $[\lambda, \lambda]$
and $d^*\kappa$ vanish on $X_{3,0}$ and $X_{2,1}$
by degree reasons, thus we need only check the condition
on $X_{1,2}$ and $X_{0,3}$.

We first calculate $2 d^*(\kappa)$ on $X_{1,2}$.
Since $\kappa$ is nonzero only on $X_{0,2}$, a direct calculation
shows that 
$$
    d^*(\kappa)( 1\ot g\ot v\wedge w\ot 1) =
    g \kappa(v,w) - \kappa ( {}^gv, {}^gw) g 
$$
for all $v,w$ in $V$ and $g$ in $G$.
Next we calculate $[\lambda,\lambda]=\phi_3^*[\mu_1,\mu_1]$ on such input. We
apply~(\ref{phi3gvw}): 
$$
\begin{aligned}
\frac{1}{2} [\mu_1,\mu_1]& \big(\phi_3(1\ot g\ot v\wedge w\ot 1)\big) \\
& =  
\mu_1\big(\mu_1(g\ot v - {}^gv\ot g)\ot w\big) 
- \mu_1\big(\mu_1(g\ot w - {}^gw\ot g)\ot v\big)\\
&\hspace{.3cm} + \mu_1\big( {}^gv\ot \mu_1(g\ot w - {}^gw\ot g)\big) 
- \mu_1\big( {}^gw\ot \mu_1(g\ot v - {}^gv\ot g)\big)\\
& \hspace{.3cm}- \mu_1\big(g\ot \mu_1(v\ot w - w\ot v)\big) 
+\mu_1\big(\mu_1( {}^gv\ot {}^gw - {}^gw\ot {}^gv)\ot g)\big)\ .
\end{aligned}
$$
Since $\mu_1 = \psi_2^*(\lambda)$ and $\lambda$ vanishes on
$X_{0,2}$, the last two summands above are zero.
We use~(\ref{psi2gv}) to simplify the remaining terms:
$$
\begin{aligned}
\mu_1 \big(\lambda(g,v)& \ot w\big)-\mu_1\big(\lambda(g,w)\ot v\big) 
    +\mu_1\big({}^gv\ot \lambda(g,w)\big) - \mu_1\big({}^gw\ot \lambda(g,v)\big)\\
  & = \sum_{h\in G}\  \lambda_h(g,v)\mu_1(h\ot w) - \lambda_h(g,w)\mu_1(h\ot v)\\
  &  \hphantom{xxxxx}  
      +\lambda_h(g,w)\mu_1({}^gv\ot h) - \lambda_h(g,v)\mu_1({}^gw\ot h)\\
& = \sum_{h\in G} \lambda_h(g,v)\mu_1(h\ot w - {}^gw\ot h ) +\lambda_h(g,w)
   \mu_1( {}^gv\ot h - h\ot v )\\
 & =\sum_{h\in G}\ \lambda_h(g,v)\mu_1( h\ot w  -  {}^hw\ot h  + 
    {}^hw\ot h  - {}^gw\ot h)\\ 
  &  \hphantom{xxxx}  
   +\lambda_h(g,w)\mu_1( {}^gv\ot h -{}^hv\ot h + {}^hv\ot h
    - h\ot v) .
\end{aligned}
$$
We rewrite $\mu_1=\psi_2^*(\lambda)$ and apply 
(\ref{psi2gv}) 
to this last expression:
$$
  \sum_{h\in G}\big( \lambda_h(g,v)\lambda(h,w) - \lambda_h(g,w)\lambda(h,v)\big)
   = \lambda\big(\lambda(g,v),w\big) - \lambda\big(\lambda(g,w),v\big).
$$
Thus, $[\lambda, \lambda]=2 d^*(\kappa)$ on $X_{1,2}$
if and only if Condition~\ref{firstobstruction12} of 
Theorem~\ref{thm:PBWconditions} holds.

Next we calculate $[\lambda,\lambda]$
and $d^*(\kappa)$ on $X_{0,3}$. 
On one hand, 
$$
 [\mu_1,\mu_1] \big( \phi_3(1\ot u\wedge v\wedge w\ot 1)\big) 
=  0\ $$
for all $u,v,w$ in $V$ by a degree argument.
On the other hand,
$$
  d^*(\kappa) (1\ot u\wedge v\wedge w\ot 1) =
    u\kappa (v,w) - \kappa(v,w)u - v\kappa(u,w) + \kappa(u,w) v
     + w \kappa(u,v) - \kappa(u,v) w .
$$
We express $\kappa$ in terms of its components $\kappa_g$ 
to see that
$[\lambda, \lambda]=2 d^*(\kappa)$ on $X_{0,3}$
if and only if Condition~\ref{firstobstruction03} of 
Theorem~\ref{thm:PBWconditions} holds. 
Thus, $[\lambda, \lambda]=2 d^*(\kappa)$ on $X_{3}$
if and only if 
Conditions~\ref{firstobstruction12} and \ref{firstobstruction03} 
of Theorem~\ref{thm:PBWconditions} hold.

\subsection{The second obstruction to integrating}
We finally show that the condition
$$[\lambda,\kappa] = 0$$
corresponds
to~(\ref{obst2}),
the condition that the second obstruction to lifting 
(or integrating) an infinitesimal deformation
must vanish.

We first express $[\lambda,\kappa]=\phi^*[\mu_1,\mu_2]$ as a map on $X_3$. 
By degree considerations,
$[\lambda, \kappa]$ vanishes everywhere except 
on the summand $X_{0,3}$ of $X_3$ and it thus suffices
to evaluate $[\mu_1,\mu_2]$ on $\phi_3(1\ot v_1\wedge v_2\wedge v_3\ot 1)$
for $v_i$ in $V$. 
In fact, we need only compute the values of the first two terms of the right
side of~(\ref{obst2}), since the other terms will be 0 on this input.  
We obtain
$$
\begin{aligned}
 &\phantom{=}
 \sum_{\sigma\in\Sym_3} \sgn \sigma \Big(\mu_1\big(\mu_2(v_{\sigma(1)} \ot
   v_{\sigma(2)})\ot v_{\sigma(3)}\big) -\mu_1\big(v_{\sigma(1)}\ot \mu_2(
   v_{\sigma(2)}\ot v_{\sigma(3)})\big)\Big)\\
 & = \sum_{\sigma\in\Alt_3} \mu_1\big(\mu_2(v_{\sigma(1)}\ot v_{\sigma(2)}
   -v_{\sigma(2)}\ot v_{\sigma(1)})\ot v_{\sigma(3)}\big)\\
& \hphantom{xxxxxxx}
  - \mu_1\big(v_{\sigma(1)}\ot \mu_2(v_{\sigma(2)}\ot v_{\sigma(3)}
   -v_{\sigma(3)}\ot v_{\sigma(2)})\big)\\
 & = \sum_{\sigma\in\Alt_3} \mu_1\big(\kappa(v_{\sigma(1)},v_{\sigma(2)})\ot
  v_{\sigma(3)} \big) 
- \mu_1\big(v_{\sigma(1)}\ot \kappa(v_{\sigma(2)},v_{\sigma(3)})\big)\\
 & = \sum_{\sigma\in\Alt_3}\sum_{g\in G}\ \kappa_g(v_{\sigma(1)},
  v_{\sigma(2)}) \big(\mu_1(g\ot v_{\sigma(3)}) -\mu_1(v_{\sigma(3)}\ot g)\big)\\
&= \sum_{\sigma\in\Alt_3}\sum_{g\in G}\ 
    \kappa_g(v_{\sigma(1)},v_{\sigma(2)}) 
     \big(\mu_1(g\ot v_{\sigma(3)}) - \mu_1({}^gv_{\sigma(3)}\ot g)\big)\\
&\hphantom{xxxxxxxx}
 +     \kappa_g(v_{\sigma(1)},v_{\sigma(2)}) 
    \big(\mu_1( {}^g v_{\sigma(3)}\ot g) -\mu_1(v_{\sigma(3)}\ot g)\big)\ .
\end{aligned}
$$
But $\mu_1=\psi_2^*(\lambda)$ and $\psi_2(1\ot v\ot g\ot 1 
-1\ot {}^gv\ot g\ot 1) = 0$, while
$$\psi_2(1\ot g\ot v\ot 1 - 1\ot {}^gv\ot g\ot 1)=1\ot g\ot v\ot 1$$
for all $v$ in $V$ and $g$ in $G$, both by~(\ref{psi2gv}).
Thus the sum is just
$$
  \sum_{\sigma\in\Alt_3}\sum_{g\in G} \kappa_g(v_{\sigma(1)},
   v_{\sigma(2)})\lambda(g,v_{\sigma(3)}) =
  \sum_{\sigma\in\Alt_3} \lambda\big(\kappa(v_{\sigma(1)},v_{\sigma(2)}),
  v_{\sigma(3)}\big)\ .
$$
Hence, $[\lambda, \kappa] = 0$
as a cochain if and only if
Condition~\ref{mixedbracket}
of Theorem~\ref{thm:PBWconditions} holds.
\end{proof}

We close with a few observations about
a central question in deformation theory:
Which Hochschild cocycles of an algebra $A$
lift (or integrate) to deformations, i.e., when is a Hochchild 2-cocycle 
the first multiplication map of some deformation of $A$?  
In characteristic zero, we showed
that every Hochschild 2-cocycle of $A=S(V)\# G$ 
of graded degree $-2$ lifts to a deformation, 
in fact, to a Drinfeld Hecke algebra (also called
a ``graded Hecke algebra'').
See~\cite[Theorem~8.7]{SheplerWitherspoon1}
for the original argument
and~\cite[Section~11]{SW-Brackets} for
additional discussion.
We give two analogs
in arbitrary characteristic
implied by Theorem~\ref{thm:PBWcohomologyconditions},
now with two parameters corresponding to two
different types of relations that may define deformations.

But first recall that every deformation of an algebra
arises from a noncommutative Poisson structure.  As the converse
is false in general, we often
seek a description of those noncommutative Poisson structures
that actually lift to deformations.
Indeed, the conditions~(\ref{barhomologicalconditions}) are in general 
necessary but not sufficient for a multiplication map to arise
from a deformation.  
However in some cases, such as when the algebra
in question is a skew group algebra $S(V)\# G$,
we may interpret these classical obstruction conditions
on a different resolution and
obtain a complete characterization of multiplication maps that 
lift to certain graded deformations, as seen in the last theorem.
Note that by focusing on graded
deformations, we restrict attention to Hochschild 2-cocycles
of graded degree $-1$. The following straightforward
corollary of Theorem~\ref{thm:PBWcohomologyconditions} explains
which Hochschild 2-cocycles can be lifted to  graded deformations
of $S(V)\# G$ isomorphic to algebras $\cH_{\lambda,\kappa}$ with
the PBW condition.
\begin{cor}
Let $k$ be a field of arbitrary characteristic.
Let $\ld$ be a noncommutative Poisson structure on $S(V)\#G$
of degree $-1$ as a graded map, and let $\kappa$ be a 2-cochain 
for which $[\lambda, \lambda]=2d^*(\kappa)$. 
As cochains on the resolution $X_{\DOT}$,
suppose the Gerstenhaber bracket 
$[\ld,\kappa]$ is zero,
and suppose that $\ld$ vanishes on $X_{0,2}$.
Then $\ld$ lifts to a graded deformation of $S(V)\#G$.
\end{cor}
\begin{proof}
Since $\deg \ld = -1$, the $A$-module map
$\ld$ takes $X_{1,1}=A\ot kG \ot V \ot A$  to $kG$,
thus  defining a $k$-linear parameter function (by restricting to $1\ot kG\ot V\ot 1$), 
$$\ld: kG \ot V \rightarrow kG\, .$$ 
Since $[\ld,\ld] = 2 d^*\kappa$, the $A$-module map
$\kappa$ has degree $-2$ and thus 
takes $X_{0,2}=1\ot V \wedge V \ot 1$ to $kG$,
defining an alternating $k$-linear
parameter function
$$\kappa: V \ot V \rightarrow kG\, .$$
By Lemma~\ref{extendparameters}, the parameters
$\ld$ and $\kappa$ extend uniquely to  cochains on $X_{\DOT}$
of the same name,
and since 
$$d^*(\ld) =0,\quad
[\ld,\ld] = 2 d^* \kappa, \quad\text{and}\quad [\ld, \kappa] = 0, $$
the algebra $\cH_{\ld,\kappa}$ satisfies the PBW condition by 
Theorem~\ref{thm:PBWcohomologyconditions}.

By Proposition~\ref{prop:deformation},
$\cH_{\lambda,\kappa}$ is isomorphic to the fiber
at $t=1$ of a deformation $A_t$ of $A=S(V)\# G$ with first
multiplication map
$\mu_1$ given by $\mu_1=\psi^*(\lambda)$, i.e.,
$$\lambda(g,v)=\mu_1(g\ot v)-\mu_1(\ ^gv\ot g)$$
for all $v$ in $V$ and $g$ in $G$.
Thus the cocycle $\lambda$, whose realization on the bar complex
is just $\mu_1$, lifts to a deformation of $S(V)\# G$.
\end{proof}

Next we generalize~\cite[Theorem~8.7]{SheplerWitherspoon1}
to fields of arbitrary characteristic
(see also~\cite[Theorem~11.4]{SW-Brackets}).

\begin{cor}\label{gradedHeckealgebras}
Let $k$ be a field of arbitrary characteristic.
Let $\kappa$ be a Hochschild 2-cocycle for $S(V)\# G$
of graded degree $-2$.
Then $\kappa$ lifts to a deformation of $S(V)\# G$.
\end{cor}
\begin{proof}
As in the last proof, we may regard
$\kappa$ as an alternating $k$-linear parameter
map, $\kappa:V\ot V\rightarrow kG$.
Then the conditions of 
Theorem~\ref{thm:PBWcohomologyconditions} are satisfied 
with $\lambda \equiv 0$
and thus $\kappa$ defines an algebra $\cH_{0,\kappa}$ with
the PBW condition. 
By Proposition~\ref{prop:deformation},
$\cH_{0,\kappa}$ is isomorphic to the fiber at $t=1$
of a graded deformation $A_t$ of $A=S(V)\# G$ with 
first multiplication map zero and second
multiplication map
$\mu_2$ given by $\mu_2=\psi^*(\kappa)$, i.e.,
$$\kappa(v,w)= \mu_2(v\ot w) - \mu_2(w\ot v)\, $$
for all $v,w$ in $V$.
We replace $t^2$ by $t$ in the graded deformation
$A_t$ to obtain an (ungraded) deformation
with first multiplication map defined by $\kappa$.
\end{proof}


\end{document}